\documentclass{amsart}
\usepackage[pagebackref]{hyperref}
\usepackage{amssymb}

\newcommand{\C}{{\mathbb C}}
\newcommand{\Q}{{\mathbb Q}}
\newcommand{\R}{{\mathbb R}}
\newcommand{\N}{{\mathbb N}}
\newcommand{\Z}{{\mathbb Z}}
\newcommand{\OO}{\mathcal O}
\newcommand{\fp}{\mathfrak p}

\newcommand{\K}{\overline{K}}
\newcommand{\Ft}{\widetilde{F}}
\newcommand{\disc}{\operatorname{disc}}
\newcommand{\Tr}{\operatorname{Tr}}
\newcommand{\Orb}{\operatorname{Orb}}
\newcommand{\Gal}{\operatorname{Gal}}
\newcommand{\eps}{\varepsilon}
\newcommand{\la}{\langle}
\newcommand{\ra}{\rangle}
\newcommand{\lra}{\longrightarrow}
\newcommand{\too}{\longmapsto}
\newcommand{\ts}{\textstyle}

\newtheorem{thm}{Theorem}
\newtheorem{prop}[thm]{Proposition}
\newtheorem{lem}[thm]{Lemma}
\newtheorem{cor}[thm]{Corollary}

\theoremstyle{definition}

\newtheorem{rem}{Remark}
\newtheorem{conjecture}{Conjecture}

\begin{document}
\title{The Euclidean Algorithm in Cubic Number Fields}
\author{Stefania Cavallar, Franz Lemmermeyer}

\abstract 
In this note we present algorithms for computing Euclidean
minima of cubic number fields; in particular, we were able
to find all norm-Euclidean cubic number fields with
discriminants $-999 < d < 10^4$. 
\endabstract

\maketitle

\section{Introduction}
This article deals with the problem of determining whether a
given cubic number field is Euclidean with respect to the
absolute value of the norm. The corresponding problem for
quadratic number fields was solved in 1952, when Barnes
and Swinnerton-Dyer showed (after much work done by
various authors) that the following list of discriminants
of norm-Euclidean quadratic number fields is complete:
$$ d \ = \ -11, -8, -7, -4, -3, 5, 8, 12, 13, 17, 21, 24, 28, 29,
   33, 37, 41, 44, 57, 73, 76. $$
In the cubic case, Davenport proved that the number of
norm-Euclidean complex cubic number fields (i.e. cubic 
fields with unit rank $1$) is finite, whereas
Heilbronn conjectured that there are infinitely
many totally real cubic fields which are norm-Euclidean.
We hope that the methods presented in this paper will 
eventually lead to a complete list of norm-Euclidean
complex cubic fields, and that extended
computations for real cubic fields will show whether
Heilbronn's conjecture is reasonable or not. 

\section{Notation}

In order to present our method we need a few definitions. 
Let $K$ be a number field, and let $\OO_K$ denote its ring of
integers. The Euclidean minimum of $\xi \in K$ is defined to be
$$ M(K,\xi) = \inf \ \{ |N_{K/\Q} (\xi-\eta)|: \eta \in \OO_K \}.$$
The field $K$ is Euclidean with respect to the absolute value
of the norm (norm-Euclidean for short) if $M(K,\xi) < 1$ for
all $\xi \in K$. Let us introduce the {\em Euclidean minimum} $M(K)$
of $K$ by putting $M(K) = \sup \ \{M(K,\xi): \xi \in K\}$. Obviously,
$K$ is norm-Euclidean if $M(K) < 1$, and not norm-Euclidean if 
$M(K) > 1$ or if there is a $\xi \in K$ such 
that $M(K) = M(K,\xi) = 1$.  

Next we introduce the inhomogeneous minimum. To this
end let $K$ be a number field generated by a root $\alpha$
of an irreducible monic polynomial $f \in \Z[x]$. Let $\alpha_1,
\ldots, \alpha_r$ denote the real roots, and
$\alpha_{r+1}, \overline{\alpha_{r+1}}, \ldots, 
 \alpha_s, \overline{\alpha_s}$ the $s$ pairs of complex
conjugate roots of $f$ in $\C$; then
the maps $\alpha \lra \alpha_j$ can be extended to 
yield $r$ embeddings $\phi_1, \ldots, \phi_r: K \lra \R$ 
and $s$ pairs of complex conjugate embeddings
$\phi_{r+1}, \overline{\phi_{r+1}}, \ldots, 
 \phi_s, \overline{\phi_s}: K \lra \C$. 

Choose a $\Q$-basis $\{\beta_1, \ldots, \beta_n\}$ of $K$;
the map
$$ \pi : K \lra \R^n: \sum a_j\beta_j \too (a_1, \ldots, a_n)$$
embeds $K$ into $\R^n$, and we will identify $K$ and $\pi(K)$
for the rest of this article. Clearly $K$ is dense in 
$\R^n$, so we will write $\K = \R^n$ if we want to make it  
clear that we regard $\R^n$ as the closure of $K$. Now
$$N: \R^n \lra \R: (x_1, \ldots, x_n) \too 
	\prod_{j=1}^r \Big|\sum_i x_i \phi_j(\beta_i)\Big| \,\cdot\,
	\prod_{j=r+1}^s \Big|\sum_i x_i \phi_j(\beta_i)\Big|^2$$
is a continuous map which coincides with the absolute value
of the norm  $N_{K/\Q}$ when restricted to $K$. Similarly,
the maps
$$ |\cdot|_j: \R^n \lra \R: x \, = \, (x_1, \ldots, x_n) \too
	|x|_j \, = \,\Big|\sum_i x_i \phi_j(\beta_i)\Big|, 
		    \quad (0 \le j \le r+s), $$
are continuous and their restrictions to $K$ agree with the 
$r+s$ archimedean valuations of $K$. By an abuse of language,
we will refer to the maps $N$ and $|\cdot|_j$ as the 'norm'
and the 'valuations' of $\K$, respectively, although $N$ is
not a norm on $\K$ since $N(x) = 0$ does not imply $x = 0$.
Similarly, the $|\cdot|_j$ are not valuations of $\K$ for the
same reason. All we can say is

\begin{prop}\label{P0}
Let $K$ be a number field, and assume that 
$\lim_{i \to \infty} |\xi_i|_j = 0$ for all $1 \le j \le r+s$
and a sequence of elements $\xi_i \in \K$. Then
$\lim_{i \to \infty} \xi_i = 0$.
\end{prop}

\begin{proof}
$K$ is an $n$-dimensional $\Q$-vector space, with a 
nondegenerate bilinear form given by
$\langle \xi, \eta \rangle = \Tr_{K/\Q}(\xi\eta)$,
where $\Tr_{K/\Q}$ denotes the trace of $K/\Q$. 
Choose  a $\Q$-basis $\{\alpha_1, \ldots, \alpha_n\}$
of $K$, and let $\{\beta_1, \ldots, \beta_n\}$
denote the dual basis with respect to $\langle\cdot,\cdot\rangle$,
i.e. the basis with the property $\Tr_{K/\Q}(\alpha_i\beta_j)
= \delta_{ij}$ (Kronecker's delta). 

Now assume that $\lim_{i \to \infty} |\xi_i|_j = 0$ for all 
$1 \le j \le r+s$, and let $\delta > 0$ be given; then there 
exists an $N \in \N$ such that $|\xi_i|_j < \delta$ for
all $i \ge N$ and $1 \le j \le r+s$. Write $\xi_i$ as
$ \xi_i = x^{(i)}_1\alpha_1 + \ldots + x^{(i)}_n\alpha_n$.
Then $ |x^{(i)}_k| \, = \, |\Tr_{K/\Q}(\beta_k\xi_i)|.$
Since the trace is the sum of all conjugates of $\beta_k\xi_i$,
applying the triangle inequality yields
$$ \begin{array}{rcl}
|x^{(i)}_k| & \le & |\beta_k\xi_i|_1 + \ldots 
	+ |\beta_k\xi_i|_r + 2 |\beta_k\xi_i|_{r+1} + \ldots
	+ 2 |\beta_k\xi_i|_{r+s} \\
	& \le & \delta(|\beta_k|_1 + \ldots + 2|\beta_k|_{r+s}) 
		< \delta C,	\end{array} $$
where $C = |\beta_k|_1 + \ldots + 2|\beta_k|_{r+s}$ does not
depend on $i$ or the choice of $\delta$. Since we can
make $\delta$ as small as we please, we find that
$\lim_{i \to \infty} x^{(i)}_k = 0$ for all $1 \le k \le n$,
and this is equivalent to $\lim_{i \to \infty} \xi_i = 0$.
\end{proof}

Obviously $K$ is norm-Euclidean if and only if for every 
$\xi \in K$ there is an $\alpha \in \OO_K$ such that 
$N(\xi - \alpha) < 1$. Actually, all known examples of 
norm-Euclidean number fields satisfy the stronger condition 
that for every $\xi \in \K$ there exists an $\alpha \in \OO_K$ 
such that $N(\xi - \alpha) < 1$. We put
$$M(\K,\xi) = \inf \ \{N(\xi - \eta): \eta \in \OO_K\},$$
and define the {\em inhomogeneous minimum} of $\K$  as
$ M(\K) = \sup \ \{M(\K,\xi): \xi \in \K\}.$
Clearly $M(K) \le M(\K)$; it is conjectured that $M(K) = M(\K)$
for all number fields, but so far this equality has been proved
only for fields with unit rank $\le 1$.

We say that $M(\K)$ is {\em isolated} if 
$M_2(\K) = \sup \ \{M(\K,\xi): \xi \in \K \setminus 
	\{\zeta: M(\K,\zeta) = M(\K)\}\} < M(\K)$. In this case,
we call $M(\K) = M_1(\K)$ the first and $M_2(\K)$ the second
minimum of $\K$.

\begin{rem}
$\K$ is a ring. In fact, the product of
two elements $\xi = \sum a_i \alpha^i$ and $\eta = \sum b_i \alpha^i$
of $K$ has the form $\sum c_i \alpha^i$, where $c_i$ is a polynomial 
in $\Q[a_1, \ldots, b_n]$. Thus, the product $\xi \eta$ can
be given a meaning for real values of the coefficients
$a_i, b_i$. 
\end{rem}

Now put $n = (K:\Q)$, and choose an integral basis 
$\{\beta_1 = 1, \beta_2, \ldots, \beta_n\}$ and a real 
number $k > 0$ (for example $k = 0.99$). We start by dividing
$$ \overline{F}_+ = \Big\{ \xi = \sum_{i=1}^n a_i\beta_i\, | \, 
     a_1 \in [0,1/2], \, a_2, \dots, a_n  \in (-1/2,1/2] \Big\}$$
into smaller subcubes. Such a subcube $S$ is called {\em $k$-covered}
(or simply {\em covered} if the reference to $k$ is clear) 
if we can find a $\gamma \in \OO_K$ such that
$N(\xi-\gamma) < k$ for all $\xi \in S$; it is called
{\em uncovered} if we cannot find such a $\gamma$ (even if there
exists one). Finally, a point $\xi \in \K$ is called
{\em $k$-exceptional} if $N(\xi-\gamma) \ge k$ for all $\gamma \in \OO_K$.

\begin{rem} Observe that $ \overline{F}_+$ is
`half a fundamental domain' in the sense that every 
$\xi \in \K/\OO_K$ has a representative in 
$\overline{F} = \overline{F}_+ \cup \overline{F}_-$,
where $\overline{F}_- = - \overline{F}_+$ (and only one unless
the representative lies on the boundary of $\overline{F}$). 
It is clearly sufficient to consider $\overline{F}_+$ since
$N(-\xi) = N(\xi)$. For cyclic cubic fields $K$ we could reduce
$\overline{F}_+$ further by exploiting the fact that
$N(\xi^\sigma) = N(\xi)$ for all $\sigma \in \Gal(K/\Q)$.
\end{rem}

\begin{rem}\label{R3}
Occasionally it simplifies computations to use 
fundamental domains other than $\overline{F}$;
they will be denoted by $\Ft$, and we 
will always assume that $\Ft$ has compact 
closure. As an example, take
$$\Ft = \Big\{ \xi = \sum_{i=1}^n a_i\beta_i\, | \, 
     a_1 \in [0,1), \, a_2, \dots, a_n  \in (-1/2,1/2] \Big\}.$$
\end{rem}

\begin{rem}
For real numbers $k', k > 0$ such that $k' > k$ it is clear
that any $k$-covered set is also $k'$-covered, and that
any $k'$-exceptional point is $k$-exceptional.
\end{rem}

\section{The Algorithms}

In this section we will describe the five programs 
({\tt Eu3\_1},\ldots, {\tt Eu3\_5}) which have allowed 
us to compute Euclidean minima $M(K)$ for many cubic number 
fields $K$. Since we used floating point arithmetic to
compute the $M(K)$ (which are rational numbers, at least
in each case we succeded in its computation), a few explanations
are in order. Suppose that we want to show $M(K) = c$ for
a number field $K$. Then we choose $k \le 0.99c$ (as a 
protection against rounding errors), and, using the programs 
{\tt Eu3\_1} -- {\tt Eu3\_3}, we compute cubes $S_j$ which 
contain all $k$-exceptional points (we want to be sure that
they contain every $c$-exceptional point). Then we exploit 
the action of the unit group $E_K$ on these cubes to compute the 
possible exceptional points, and since this is done 
with integer arithmetic, we are able to get exact results.

Experiments with e.g. fields whose minima are known from
hand computations (choosing values of $k$ close to the
minimum and using very small cube lengths $\ell$)
have led us to trust our results. Moreover, all our
results agree with those obtained before (e.g. by Smith \cite{Sm} 
and Taylor \cite{T}) or during the writing of this article (by 
D. Clark \cite{Cla} and R. Qu\^eme). 

Our programs require as input a file called {\tt disc\/}
(i.e. {\tt 985} for the field with discriminant $d = 985$;
for complex fields we used the absolute value of the 
discriminant preceded by "\_", for example {\_{\tt 199}} 
for the field with discriminant $d = -199$). 
This file contains the following data:
\begin{itemize}
\item $\disc K$;
\item the coefficients of the irreducible monic polynomial $f$;
\item
   \begin{itemize}
   \item for real fields the roots of the polynomial $\alpha$, $\alpha'$, 
         $\alpha''$; 
   \item for complex fields the real root $\alpha$, the real and the imaginary 
         part of the complex root; 
   \end{itemize}
\item the coefficients with respect to the base 
         $\{\frac{1}{g},\frac{\alpha}{g}, \frac{\alpha^2}{g}\}$ 
         (where $g$ is defined below) of a system of independent units;
\item the index $g=(\OO_K:\Z [\alpha])$ and in the case of $g\not=1$ also
      $g_x$, $g_y$, $g_z$, where 
$\{1, \alpha, \theta=\frac{g_x}{g}+\frac{g_y}{g}\alpha+\frac{g_z}{g}\alpha^2\}$
      form a $\Z$-basis of $\OO_K$;
\item the value $k>0$;
\item the edge length $\ell$ of the cubes;
\item the coordinates of the uncovered cubes.
\end{itemize}

For the sake of simplicity we treat only
cubes having the same size; thus a cube is
uniquely determined by its leftmost corner and
the edge length $\ell$. 
We therefore start with the four cubes making up
$\overline F_+$, and the initial file {\tt 985}, 
for example, looks as in Table 1 below:

\medskip

\hbox{ 
\begin{tabular}{rrrr}
\multicolumn{4}{c}{Table 1} \\
\multicolumn{4}{c}{} 	     \\
\tt 985    &  \tt1      & \tt -6      & \tt  -1 \\
  \multicolumn{4}{r}{\tt -2.93080160017276} \\
  \multicolumn{4}{r}{\tt -0.16296185677753} \\
  \multicolumn{4}{r}{\tt  2.09376345695029 } \\
& \tt 0       & \tt 1       & \tt 0 \\
& \tt 2      & \tt -1       & \tt 0 \\
&        &  \tt 1 & \\
&       & \tt 0.9 & \\
&       & \tt 0.5 & \\
& \tt 0      & \tt  -0.5   & \tt  -0.5 \\
& \tt 0      & \tt  -0.5   & \tt   0 \\
& \tt 0      & \tt   0     & \tt  -0.5 \\
& \tt 0      & \tt   0     & \tt   0   \end{tabular} 

\hfill

\begin{tabular}{rrrr}
\multicolumn{4}{c}{Table 2} \\
\multicolumn{4}{c}{} 	     \\
\tt 985	 & \tt 1    & \tt -6	& \tt -1 \\
\multicolumn{4}{r}{\tt -2.93080160017276} \\
\multicolumn{4}{r}{\tt -0.16296185677753} \\
\multicolumn{4}{r}{\tt  2.09376345695029} \\
& \tt 0	& \tt  1     & \tt 0 \\
& \tt 2	& \tt -1     & \tt 0 \\
& 	& \tt 1   &  \\
& 	& \tt 0.9 &  \\
& 	& \tt 0.1 &  \\
&  \tt 0.3 & \tt -0.5 & \tt -0.5 \\
&  \tt 0.3 & \tt -0.5 & \tt -0.3 \\
&   \multicolumn{2}{r}{\tt  . . .} & \\
&  \tt 0.4  & \tt 0.4  & \tt 0.4 \\    \end{tabular} 
 
\hfill

\begin{tabular}{rrrr}
\multicolumn{4}{c}{Table 3} \\
\multicolumn{4}{c}{} 	     \\
\tt 985     & \tt  1     & \tt  -6      & \tt -1 \\
&   \multicolumn{2}{r}{\tt  . . .} & \\
& 	&  \tt 0.02 &  \\
& \tt 0.38 & \tt -0.22  & \tt 0.38 \\
& \tt 0.38 & \tt -0.22  & \tt 0.4  \\
& \tt 0.38 & \tt -0.2   & \tt 0.38 \\
& \tt 0.38 & \tt -0.18  & \tt 0.38 \\
& \tt 0.38 & \tt -0.2   & \tt 0.4  \\
& \tt 0.4  & \tt -0.22  & \tt 0.38 \\
& \tt 0.4  & \tt -0.22  & \tt 0.4  \\
& \tt 0.4  & \tt -0.2   & \tt 0.38 \\
& \tt 0.4  & \tt -0.18  & \tt 0.38 \\
& \tt 0.4  & \tt -0.2   & \tt 0.4  \\  \end{tabular}  

}

\medskip

\noindent
We now run the programs  {\tt Eu3\_1}, {\tt Eu3\_2} and 
{\tt Eu3\_3}, which will be described in the sequel, 
on the file {\tt disc\/}. The programs {\tt Eu3\_4} and 
{\tt Eu3\_5} will be explained before Prop. \ref{P4} and 
Cor. \ref{C1}, respectively.

\subsection*{{\tt Eu3\_1}} \label{scansione}
This program first asks for a discriminant, then reads the 
corresponding file {\tt disc\/}. The first eight lines of 
{\tt disc\/} are copied to the (temporary) file {\tt disc.new\/}.
The next input is an integer $f$ which is the factor
by which we divide the edge length of the cubes.  
We have used $f \in \{1, 2, 4, 5\}$, depending on the size
of {\tt disc\/}; of course the choice $f = 1$ is only useful
after $k$ has been replaced by some $k' > k$.
Thus {\tt Eu3\_1} reads $\ell$ from {\tt disc\/} and writes 
$\ell/f$ to {\tt disc.new\/}. Moreover, if no file 
{\tt disc.p\/} exists, {\tt Eu3\_3} creates one and writes
the translation vector {\tt (0, 0, 0)} into it.

Now {\tt Eu3\_1} splits each cube read from {\tt disc\/}
into $f^3$ smaller ones, computes an upper bound $B = B(S)$
of the minimum for each of these subcubes, and writes
those $S$ with $B(S) \ge k$ to the file {\tt disc.new\/}.
Having reached the end of the file {\tt disc\/}, it copies
{\tt disc\/} to {\tt disc.bak\/} (a security backup) and 
{\tt disc.new\/} to {\tt disc\/}.

How do we bound the minimum on a cube 
$S=[a,a+\ell]\times[b,b+\ell]\times[c,c+\ell]$?
Since the norm is the product of the three $\K$-valuations,
we only need to find bounds for $|\xi|_j$, where 
$\xi = x + y\alpha + z\theta \in S$. But since $|\xi|_j$ is 
a linear function of $x, y, z$, it takes its maximum at
the corners of the subcubes. Instead of computing $|\cdot|_j$
at all eight corners and taking the maximum of these 
values as our upper bound (as the programs with which we computed
the tables at the end did), we 
can use a trick due to Roland Qu\^eme \cite{Q} which 
gives this bound at one stroke (but which doesn't seem to work
except for valuations corresponding to real embeddings):
in fact, the value of $|\cdot |_j$ at the eight corners is one of
$$|\xi_0 \pm \ts{\frac{\ell}2} 
	\pm \frac{\ell}2 \alpha \pm \frac{\ell}2 \theta|_j,
	\qquad \text{ where } \quad \xi_0 = 
	(a + \ts{\frac{\ell}2}) + (b + \ts{\frac{\ell}2})\alpha +
  	(c + \ts{\frac{\ell}2})\theta,$$
the corners corresponding to the different choices of the
signs. Using the triangle inequality we easily get 
$$|\xi_0 \pm \ts{\frac{\ell}2} \pm \frac{\ell}2 \alpha 
	\pm \frac{\ell}2 \theta|_j \ \le \
  |\xi_0|_j + \frac{\ell}2(1+|\alpha|_j+|\theta|_j).$$ 
On the other hand, choosing the signs in such a way that
$\xi_0$, $\pm 1$, $\pm \alpha$ and $\pm \theta$ all have
the same sign (in the embedding corresponding to $|\cdot |_j$; here 
is where we need that the embedding is real)
we see that this bound is best possible. 

Using this method of bounding $N(S-\gamma)$, {\tt Eu3\_1}
will start looking for translation vectors in {\tt disc.p};
if there is no element in this file such that $N(S-\gamma) < k$
we search for translation vectors in the set 
\begin{equation}\label{setI}
   I = \{x+y\alpha+z\theta | (x,y,z)\in\Z^3, 
		|x|\leq M_x, |y|\leq M_y, |z|\leq M_z\}, \end{equation}
where $M_x, M_y, M_z$ were usually chosen (depending on $\ell$)
as follows:
\smallskip
\begin{center}
\begin{tabular}{|c|c|c|c|} \hline
 & $\ell \ge 0.02$ & $0.01 \ge \ell \ge 0.001$ & $\ell \le 0.0005$ \\ \hline
 $M_x$ &  $8$  & $19$ & $30$ \\
 $M_y$ &  $5$  & $12$ & $17$ \\
 $M_z$ &  $2$  & \phantom{1}$3$ &  \phantom{1}$5$ \\ \hline \end{tabular} 
\end{center} 
\smallskip
\noindent
The nonsymmetric limits were suggested by experience. 
If we find a translation vector $\gamma \in \OO_K$ such that 
$N(S-\gamma) < k$ we write $\gamma$ to the file {\tt disc.p\/}.
There are several reasons for proceeding like this:
\begin{enumerate}
\item If $N(S-\gamma) < k$, then $\gamma$ has a good
      chance of satisfying $N(S'-\gamma) < k$ for cubes
      $S'$ in the vicinity of $S$. By searching 
      {\tt disc.p\/} first we actually save much CPU time.
\item If we find that we have to replace $k$ by some $k' < k$,
      we have to redo the computations from start; we are,
      however, able to use the translation vectors found in
      the previous runs. In fact, our programs allow the
      option of searching for new vectors or just using
      those in {\tt disc.p}.
\end{enumerate}

After the first run of {\tt Eu3\_1} with $f = 5$
(which took 0.63 seconds of CPU time on an RS 6000), 
the file {\tt 985} looks as in Table 2.
It contains exactly 106 noncovered cubes. The file
{\tt 985.p} contains the following translation vectors:
\begin{verbatim}
(0 0 0), (0 1 0), (0 2 1), (1 0 0), (-1 -4 -1), (0 3 -1), (-1 -3 2),	
(0 1 -1), (0 2 0), (3 0 -1), (0 2 -1), (-1 0 0), (-1 -4 2), (-2 0 0)
\end{verbatim}

\subsection*{{\tt Eu3\_2}}
This program acts like {\tt Eu3\_1} with the difference that 
the original cube is written to {\tt disc.new\/} as soon as
one of its subcubes  cannot be covered. In other words:
{\tt Eu3\_2} eliminates those cubes whose subcubes of
length $\ell/f$ can all be covered. This is convenient
if we already have to deal with a lot of cubes and a further 
division done as by {\tt Eu3\_1} is likely to lead to an enormous
number of smaller cubes. We usually run {\tt Eu3\_2} before
using {\tt Eu3\_3} in order to save CPU time.

\subsection*{{\tt Eu3\_3}}
Scanning through all the integers of $I$ (see (\ref{setI})) 
takes much time.  We can avoid searching for explicit 
translations if we proceed as follows: we multiply a cube 
$T$ with a non-torsion unit $\eps$, translate the result back 
into the fundamental domain by subtracting $\beta \in \OO_K$ 
and look whether $\eps T - \beta$ intersects one of the cubes 
not yet covered.  The program does not really compare the
oblique prism $\eps T - \beta$ with the uncovered cubes but 
rather uses the smallest box $S$ which contains $T$  and has 
faces parallel to the coordinate planes 
(this will be improved in the next version of our programs;
we also remark that -- in order to avoid rounding errors -- 
we do not compare $\eps T - \beta$ with $T$ and $-T$ but with a slightly 
larger cube obtained by adding (resp. subtracting) $\frac12 \ell$ 
to (resp. from) the coordinates of the corners of $T$).
Evidently we have to compare the box also with the `opposite' 
cubes, i.e. the cubes multiplied by $-1$, since we only kept the 
`bad' cubes of half the fundamental domain $\overline F_+$.  
If there is no intersection, the cube itself can be eliminated.

The reason why this program is so successful is the following:
suppose that $S$ is a subcube such that $\eps S - \alpha$
(where $\alpha \in \OO_K$ is the element needed to translate
$\eps S$ back into $\overline{F}$) is covered by 
$\gamma \in \OO_K$, i.e. $N(\eps S - \alpha - \gamma) < k$.
This means of course that  $N(S - \beta) < k$ for
$\beta = \eps^{-1}(\alpha+\gamma)$: but $\beta$ will
usually have much larger coefficients than those scanned in 
(\ref{setI}). Moreover, in general $\eps S - \alpha$ will 
not be covered by a single element $\gamma \in \OO_K$,
which means that we would have to divide $S$ into subcubes
before we could cover it directly.

A run of {\tt Eu3\_1} on the file in Table 2, 
again with $f = 5$, leaves only $27$ subcubes uncovered. 
Running {\tt Eu3\_3} twice on the file obtained
we are left with $10$ uncovered cubes (see Table 3).
Running {\tt Eu3\_2} on the file in Table 3 deletes 
{\tt (0.38 -0.18  0.38)} and {\tt (0.4 -0.18  0.38)}. 
Now we have covered $\overline{F}$ except for the set
$T = [0.38, 0.42] \times [-0.22, -0.18] \times [0.38, 0.42]$.
Multiplying the corners $P$ of $T$ by the unit $\alpha$
we find 
\medskip

\begin{center}
\begin{tabular}{|c|c|}\hline
 $P$ & $\alpha P$ \\ \hline
 $0.38-0.22\alpha+0.38\alpha^2$ & 
 	$0.38 - 0.34\alpha +0.40\alpha^2 + 3\alpha - \alpha^2$ \\  
 $0.38-0.22\alpha+0.42\alpha^2$ & 
	$0.42 - 0.10 \alpha + 0.36 \alpha^2 + 3\alpha - \alpha^2$ \\
 $0.38-0.18\alpha+0.38\alpha^2$ & 
	$0.38 - 0.34\alpha + 0.44 \alpha^2 + 3\alpha - \alpha^2$ \\ 
 $0.38-0.18\alpha+0.42\alpha^2$ & 
	$0.42 -0.10 \alpha + 0.40 \alpha^2 + 3\alpha - \alpha^2$ \\
 $0.42-0.22\alpha+0.38\alpha^2$ &
	$0.38 -0.30 \alpha +0.40\alpha^2 + 3\alpha - \alpha^2$ \\
 $0.42-0.22\alpha+0.42\alpha^2$ & 
	$0.42 -0.06 \alpha + 0.36 \alpha^2 + 3\alpha - \alpha^2$ \\
 $0.42-0.18\alpha+0.38\alpha^2$ & 
	$0.38 - 0.30\alpha +0.44\alpha^2 + 3\alpha - \alpha^2$ \\ 
 $0.42-0.18\alpha+0.42\alpha^2$ & 
	$0.42 -0.06 \alpha + 0.40\alpha^2 + 3\alpha - \alpha^2$ \\ \hline
\end{tabular}
\end{center}
\medskip

This shows that $\alpha T$ is contained in the set
$$T' = [0.38, 0.42] \times [-0.34,-0.06] \times 
		[0.36, 0.44] + 3\alpha - \alpha^2$$
(in general, this is a very crude estimate; the next version of 
our programs will take the actual shape of $\eps T$ into account).
Observe that $T'-3\alpha + \alpha^2$ does not intersect $-T$.
If we were to keep dividing the uncovered cubes, this picture
would not change: the length of the cubes would become smaller
and smaller, and so would the size of our uncovered set $T$,
but $\alpha T - 3\alpha + \alpha^2$ would always have
points in common with $T$. This is the situation 
which is described in

\begin{prop}\label{T2.9}
Let $K$ be a number field and $\eps$ a non-torsion unit 
of $E_K$. Suppose that $T \subset \Ft$ (vf. Rem. \ref{R3}) 
has the following property:
\begin{quotation}
there exists a unique $\beta \in \OO_K$ such that, 
for all $\xi \in T$, the element $\eps\xi - \beta$ lies 
in a $k$-covered region of $\Ft$ or again in $T$.
\end{quotation}
Then every $k$-exceptional point $\xi_0 \in T$ satisfies
$|\xi_0-\frac{\beta}{\eps-1}|_j = 0$ for every $\K$-valuation
$|\cdot|_j$ such that $|\eps|_j > 1$.
If, moreover, $|\eps|_j \ne 1$ for all $\K$-valuations, then
the sequence $\xi_0, \xi_1, \xi_2, \ldots$ of $k$-exceptional points
defined by the recursion $\xi_{i+1} = \eps\xi_i - \beta$ satisfies
$\lim_{i \to \infty} \xi_i = \frac{\beta}{\eps-1}$.
\end{prop}

\begin{proof}
Suppose that $\xi_0 \in T$ is $k$-exceptional. Since
$M(K,\xi) = M(K,\eps\xi-\beta)$, our assumption implies that 
all the $\xi_i$ defined by $\xi_{i+1} = \eps\xi_i - \beta$
are $k$-exceptional points in $T$. Now $\zeta = \frac{\beta}{\eps-1}$
is the fixed point of the map $\xi \too \eps\xi-\beta$.
Induction shows that $\xi_i - \zeta = \eps^i(\xi_0 - \zeta)$
for all $i \ge 0$. In particular we see that
\begin{equation}\label{E1}
 |\xi_i-\zeta|_j \ = \ |\eps|_j^i\,|\xi_0-\zeta|_j.
\end{equation}
Now we claim that there exists a constant $C > 0$ such
that $|\xi_i-\zeta|_j \le C$ for all $i \ge 0$ ($C = C(j)$ may
depend on $j$, but we can always choose $C$ as the maximum
of the (finitely many) $C(j)$). In fact, since
$\xi_i, \zeta \in \Ft$, we see that
their difference is an element of $2\Ft \subset \K$. 
Since $2\Ft$ has compact closure and $|\cdot|_j$ is continuous,
$|\cdot|_j$ has a maximum $C$ on the closure of $2\Ft$
and hence is bounded on $2\Ft$.

Assume that $|\eps|_j > 1$; then the fact that the
left hand side of Equ. (\ref{E1}) is bounded implies
that $|\xi_0-\zeta|_j = 0$. This in turn gives immediately
$|\xi_i-\zeta|_j = 0$. If, moreover,  $|\eps|_j \ne 1$
for all $j \le r+s$, then either $|\eps|_j > 1$ and
$|\xi_i-\zeta|_j = 0$, or  $|\eps|_j < 1$ and
$0 = \lim_{i \to \infty} |\xi_i-\zeta|_j.$
This implies $\lim \xi_i = \zeta$ by Prop. \ref{P0}. 
\end{proof}

In our case there are three embeddings; we have $|\alpha|_j > 1$
for $j = 1, 3$, and $|\alpha|_j < 1$ for $j = 2$. From
Prop. \ref{P1} we can deduce that every $0.9$-exceptional
point $\xi \in T$ satisfies
$|\xi - \zeta|_1 = |\xi - \zeta|_3 = 0$, where
$\zeta = \frac{3\alpha-\alpha^2}{\alpha-1} = 
 \frac15(2-\alpha+2\alpha^2)$. In order to show that
$\xi$ is in fact the only $0.9$-exceptional point $\xi \in T$ 
we have to prove $|\xi - \zeta|_2 = 0$; this is done by using
the inverse of the unit $\alpha$:

\begin{thm}\label{T1}
Let $K$ be a number field, $T$ a compact subset of $\Ft$,
and let $\eps \in E_K$ be a non-torsion unit. Suppose that
\begin{enumerate}
\item there is a $\beta \in \OO_K$ such that, for all 
	$\xi \in T$, $\eps\xi - \beta$ lies in a $k$-covered
	 region of $\Ft$ or again in $T$;
\item for all $\xi \in T$ there is a $\gamma \in \OO_K$
	such that $\eps^{-1}\xi - \gamma$ lies in a $k$-covered
	 region of $\Ft$ or again in $T$;
\item $|\eps|_j \ne 1$ for $1 \le j \le r+s$.
\end{enumerate}
Then $\frac{\beta}{\eps-1}$ is the only possible $k$-exceptional
point of $T$.
\end{thm}

\begin{proof}
Let $\xi \in T$ be a $k$-exceptional point. 
If $|\eps|_j > 1$ then Prop. \ref{T2.9} shows that 
$|\xi-\frac{\beta}{\eps-1}|_j = 0$. The other $\K$-valuations
$|\cdot|_j$ satisfy $|\eps|_j < 1$ because of 3., and we
see $|\eps^{-1}|_j > 1$. Since $\xi$ is $k$-exceptional, so is
$\xi_1 = \eps^{-1}\xi-\gamma \in T$. Now
$\eps\xi_1 - \beta = \eps(\eps^{-1}\xi-\gamma) - \beta
 = \xi - (\beta + \eps \gamma)$; but $\xi \in T$
and $\xi - (\beta + \eps \gamma) \in T$ imply that
$\beta + \eps \gamma = 0$, i.e. $\gamma = -\eps^{-1}\beta$.
Therefore the $\gamma$ in 2. is uniquely determined,
and we can apply Prop. \ref{T2.9} with $\eps^{-1}$ and $\gamma$
instead of $\eps$ and $\beta$. This shows that
any $k$-exceptional point $\xi \in T$ satisfies
$|\xi - \frac{\gamma}{\eps^{-1}-1}|_j = 0$ for all $\K$-valuations
with $|\eps|_j < 1$. But $\frac{\gamma}{\eps^{-1}-1} = 
\frac{\eps\gamma}{1-\eps} = \frac{\beta}{\eps-1}$.
Thus $|\xi - \frac{\beta}{\eps-1}|_j = 0$
for all $1 \le j \le r+s$, and by Prop. \ref{P0} this
implies that $\xi = \frac{\beta}{\eps-1}$. 
\end{proof}

\begin{rem}
This theorem is attributed to Cassels in \cite{BSD}.
\end{rem}

\begin{rem}
For every number field $K$ there exists a complete system of 
independent units $\eps_i$ such that $|\eps_i|_j \ne 1$. This
follows directly from Minkowski's proof of Dirichlet's
unit theorem. 
\end{rem}

\begin{rem}
If condition 1. of Thm. \ref{T1} is satisfied but 2. is not
(for example if there is a second uncovered subset $T'$ 
such that $\eps^{-1}T$ and $T'$ have common points mod $\OO_K$),
then $T$ might contain irrational exceptional points (converging
to $\frac{\beta}{\eps-1}$) as in the last sentence of 
Prop. \ref{T2.9}.
\end{rem}

\begin{rem}
If $\eps T$ intersects $-T$, try to apply Thm. \ref{T1}
with $\eps$ replaced by $-\eps$.
\end{rem}

\begin{rem}
Suppose that $(\frac12, \frac12,\frac12)$ is one of the 
$k$-exceptional points; in this case, there will be uncovered
sets in all eight corners of the fundamental domain $\overline{F}$.
In order to apply Prop. \ref{T2.9} one has to choose
$\Ft$  e.g. as in Remark \ref{R3}, because this allows us to
collect these uncovered cubes into one set $T$ lying in the center
of $\Ft$.
\end{rem}

In our example, computations similar to those above with 
$\alpha^{-1} = -6 + \alpha + \alpha^2$ instead of $\alpha$ 
show that $\alpha^{-1}T$ is contained in 
$$T'' = [0.26,0.54] \times [-0.24, -0.16] \times 
		[0.38,0.42] + 3 - \alpha.$$
Now Thm. \ref{T1} shows that the only possible $k$-exceptional 
point ($k = 0.9$) in $T$ is $\xi = \frac{3\alpha-\alpha^2}{\alpha-1} = 
	\frac25 - \frac15 \alpha + \frac25 \alpha^2$.

\medskip
\noindent 
{\tt Eu3\_4}. This program does the necessary computations:
it checks whether a cube $S$ multiplied 
by a non-torsion unit $\eps$ and translated back into the 
fundamental domain $F$ intersects either $S$ or $-S$; in 
both cases, the possible exceptional point is computed and 
written to the file {\tt disc.n}. We can also search for orbits of
length $\ge 2$ by replacing $\eps$ by $\eps^n$ for some
$n \ge 2$, provided $\ell$ is small enough.
Verifying the conditions of Thm. \ref{T1} are currently
still done by hand; see, however, Rem. \ref{R12}.
The actual computation of the possible exceptional point is done
using integer arithmetic; whenever the precision was insufficient
(e.g. for $\disc K = -680$ or $- 728$) we used PARI.
\smallskip

The next question is how to compute $M(K,\xi)$. This is 
done as follows: first we notice that $M(\K,\xi)
= M(\K,\xi+\alpha)$ for $\alpha \in \OO_K$, i.e. $M(\K,\xi)$
only depends on the coset $\xi + \OO_K$ of $\xi$ in $\K/\OO_K$. Next 
we observe that $M(\K,\xi) = M(\K,\eps\xi)$ for units $\eps \in E_K$.
For $\xi \in \K/\OO_K$ put $\Orb_\eps(\xi) = \{\eps^n\xi: n \in \Z\}$
(this is the orbit of $\xi$ under the action of $\eps$)
and  $\Orb(\xi) = \{ \eps\xi: \eps \in E_K\}$. Then the
Euclidean minimum is constant on every orbit.

\begin{prop}\label{P4}
For number fields $K$ with unit rank $\ge 1$, the following
properties of $\xi \in \K/\OO_K$ are equivalent:
\begin{enumerate}
\item[i)] $\Orb(\xi)$ is finite;
\item[ii)] $\Orb_\eps(\xi)$ is finite for some non-torsion 
	unit $\eps \in E_K$;
\item[iii)] $\xi \in K/\OO_K$.
\end{enumerate}
\end{prop}

\begin{proof}
The implication i) $\Longrightarrow$ ii) is obvious. Assume that
$\eps$ is a non-torsion unit such that $\Orb_\eps(\xi)$ is finite.
Then there exists an $n \in \N$ such that $\eps^n\xi = \xi$, and
this implies that $\xi = \frac{\alpha}{\eps^n-1} + \OO_K$ for
some $\alpha \in \OO_K$, i.e. $\xi \in K/\OO_K$.

Finally assume that  $\xi = \frac{\alpha}{\beta}+ \OO_K$ for some
$\alpha,\beta \in \OO_K$. If $\beta \in E_K$, then $\xi = 0 + \OO_K$
and the claim is trivial. Otherwise observe that multiplication
by a unit $\eps$ maps $\xi$ to some element of the
form $\frac{\alpha'}{\beta}+ \OO_K$, where $\alpha' \equiv
\alpha\eps \bmod \beta$. This shows that $\# \Orb(\xi) < 
\# (\OO_K/\beta\OO_K) = N(\beta)$. In fact, if $\OO_K$ has
class number $1$ and if $(\alpha,\beta) = 1$, then we
clearly have $\# \Orb(\xi) \le \# (\OO_K/\beta\OO_K)^\times$.
\end{proof}

In this paper we will not deal with computing minima
for $\xi \in \K/\OO_K$ with infinite orbit (the last
sentence of Prop. \ref{T2.9} gives a hint as to how such
infinite orbits might arise), so we assume
from now on that $\xi \in K/\OO_K$.
The basic idea how to compute $M(K,\xi)$ is due to 
Barnes and Swinnerton-Dyer (see \cite{BSD}, Thm. B, for the
case of real quadratic number fields).

\begin{prop}\label{P1}
Let $K = \Q(\alpha)$ be a number field with unit group $E_K$. If,
given $\xi \in K$ and a real number $k > 0$, there exists  
$\gamma \in \OO_K$ such that $N(\xi - \gamma) < k$, then 
there exists $\zeta = \sum_{j=0}^{n-1} a_j\alpha^j \in K$ 
with the following properties:
\begin{enumerate}
\item $\zeta \equiv \xi_j \bmod \OO_K$ for some $\xi_j \in \Orb(\xi)$;
\item $|a_i| < \mu_i$ ($0 \le i < n$) for some constants $\mu_i > 0$ 
	depending only on $K$;
\item $N(\zeta) < k$.
\end{enumerate}
\end{prop}
Since the number of elements of $K$ satisfying 1. and 2. 
is finite, we can prove $M(K,\xi) \ge k$ by simply 
computing the norms of all these elements. We will prove 
Prop. \ref{P1} only for cubic fields.

Let $K = \Q(\alpha)$ be a cubic number field; replace $\xi$ by 
$\xi - \gamma$ and choose a unit $\eps \in E_K$ such that the 
conjugates of $\beta = \xi \eps = a + b\alpha + c\alpha^2 $ are 
small. Let us consider the following system of equations:
$$ \begin{array}{rcl} 
   \beta   & = & a + b\alpha + c\alpha^2 \\
   \beta'  & = & a + b\alpha' + c{\alpha'}^2 \\
   \beta'' & = & a + b\alpha'' + c{\alpha''}^2  \end{array} $$
This system is linear in $a, b, c$,
and the square of its determinant is
$$ \Delta \ = \ \det \left( 
\begin{array}{ccc}
 1 & \alpha  & \alpha^2 \\
 1 & \alpha' & {\alpha'}^2 \\
 1 & \alpha'' & {\alpha''}^2 \end{array} \right)^2
  \ = \  \disc(1, \alpha, \alpha^2),$$
which is clearly $\ne 0$. In fact, we have
$\Delta = g^2 \disc K$ for some integer $g$ called the
index of $\alpha$. Therefore we get, by Cramer's rule,
$$ \begin{array}{rcl} 
  \sqrt{\Delta} a & = &
   \beta \alpha'\alpha''(\alpha''- \alpha') + 
   \beta' \alpha''\alpha(\alpha - \alpha'') + 
   \beta'' \alpha\alpha'(\alpha' - \alpha), \\ 
\sqrt{\Delta} b & = &
    \beta({\alpha'}^2 - {\alpha''}^2) + 
    \beta'({\alpha''}^2 - \alpha^2) + 
    \beta''(\alpha^2 - {\alpha'}^2), \\
\sqrt{\Delta} c & = &
    \beta(\alpha'' - \alpha') + 
    \beta'(\alpha - \alpha'') +  
    \beta''(\alpha' - \alpha). \end{array} $$ 
In order to compute bounds for $a, b$, and $c$ we have to 
find good bounds for the conjugates of $\beta = \xi \eps$.

We begin with the complex cubic case; as for real quadratic
number fields, there is one fundamental unit $\eta$. Replacing
$\eta$ by $\eta^{-1}$ if necessary we may assume that 
$|\eta| > 1$. Now for every $c_1 > 0$ there is a unit 
$\eps = \eta^m$ such that 
$$ c_1 \ < \ | \xi \eps| \ \le \ c_1 \cdot |\eta|.$$
Since $|\xi' \eps'| = |\xi'' \eps''|$ in the complex case,
we get
$$ |\xi' \eps'|^2 \ = \ \frac{N(\xi)}{|\xi \eps|} \le \frac{k}{c_1},$$
where we have put $k = N(\xi) = N(\xi \eps)$. We will choose
$c_1$ in such a way that the resulting bounds on 
$X = |\beta(\alpha'' - \alpha')|$, $Y = |\beta'(\alpha - \alpha'')|$ and
$Z = |\beta''(\alpha' - \alpha)|$ are equal. A little computation
shows $c_1^3 = \frac{k}{|\eta|^2} 
	\frac{\,|\alpha-\alpha''|^2\,}{\,|\alpha'-\alpha''|^2\,}$,
and this yields
$$ X, Y, Z \  \le \ \sqrt[3\,]{k |\eta|} \sqrt[6]{|\Delta|}.$$
Applying Lemma \ref{Lin} below to 
$x = X, y = Y$ and $z = Z$ we find $xyz = k\sqrt{|\Delta|}$ and
$$ |c| \ \le \ \sqrt[3\,]{\frac{k |\eta|}{|\Delta|}}
		\Big(2 + \frac{1}{|\eta|}\Big) .$$
The obvious bound 
$$ |b| \ \le \ (X|\alpha''+\alpha'| + Y|\alpha'' + \alpha| + 
	Z|\alpha+\alpha'|)/\sqrt{\Delta}
	\ \le \ \sqrt[3\,]{\frac{k |\eta|}{|\Delta|}}(|\alpha''+\alpha'| 
	+ |\alpha'' + \alpha| + |\alpha+\alpha'|)$$
can be sharpened by applying Lemma \ref{Lin} to 
$x = X |\alpha' + \alpha''|$,  $y = Y|\alpha''+\alpha|$
and $z = Z|\alpha+\alpha'|$, and the same goes for $|a|$.
The actual bounds coming from Lemma \ref{Lin} are computed by 
machine in each case because they depend on the size of
$|\alpha+\alpha'|, \ldots$ etc. This concludes the proof in the
complex cubic case.

\begin{rem}
The bounds in the above proof are much better than those
obtained by Cioffari \cite{Cio} for the case of pure
cubic fields.
\end{rem}

\begin{lem}\label{Lin}
Suppose that $x, y, z \in \R$ satisfy the inequalities
$0 < x \le c_1$, $0 < y \le c_2$, $0 < z \le c_3$, and
$0 < xyz = k$. Then 
$$ x + y + z \ \le \ \max_{i \ne j} 
	\ \left\{ c_i + c_j + \frac{k}{c_ic_j} \right\}.$$  
\end{lem}

\begin{proof}
We want to find bounds for $f(x,y) = x+y+\frac{k}{xy}$.
Now $f$ is positive in the domain under consideration
$$ D = \Big\{x,y > 0: x\le c_1, y \le c_2, 
		xy \ge \frac{k}{c_3} \Big\},$$
its gradient vanishes only at $x = y = \sqrt[3\,]{k}$, and 
its Hesse matrix there is positive definite; this implies that
$f$ takes its maximum on the boundary.

Assume for example that $x = c_1$; then we have to find an
upper bound for $f_1(y) = c_1 + y + \frac{k}{c_1y}$. Again,
$f_1$ assumes its maximum on the boundary. For $y = c_2$
we get the bound $c_1 + c_2 + \frac{k}{c_1c_2}$; on the
other hand from $z = \frac{k}{c_1y} \le c_3$ we get
$y \ge \frac{k}{c_1c_3}$, and we find
$f_1(\frac{k}{c_1c_3}) = c_1 + c_3 + \frac{k}{c_1c_3}$.

The cases $y = c_2$ and  $xy =  \frac{k}{c_3}$  are
treated similarly.
\end{proof}

Let $K$ be a real cubic number field, and let $\eta_1$
and $\eta_2$ denote two independent units. We will
denote the conjugates of $\xi \in K$ by $\xi, \xi', \xi''$.
For units $\eps \in E_K$ and a fixed embedding $K \lra \R$
we define 
$$ \gamma(\eps) := \left\{ 
       \begin{array}{ll} 
       |\eps|, & \mbox{ if } |\eps| \ge 1, \\
       |\eps|^{-1}, & \mbox{ if } |\eps| < 1, \end{array} \right. $$
and we put ${\gamma}_1 = \gamma(\eta_1), \ldots, \gamma'_2 = 
\gamma(\eta'_2)$.

Suppose that $\xi \in K$ has norm $k$; then, for any real numbers
$c_1, c_2 > 0$, we can find a unit $\eps \in \la \eta_1, \eta_2\ra$ 
such that (cf. \cite{LP})
$$c_1  \, < \,  |\xi \eps|  \, \le \, c_1 {\gamma}_1  {\gamma}_2, \qquad
  c_2  \, < \,  |\xi' \eps'|  \, \le \, c_2 {\gamma'_1}  {\gamma'_2}.$$ 
This gives us the following bounds on $|\xi'' \eps''|$:
$$ |\xi'' \eps''| \ = \ \frac{N(\xi \eps)}{|\xi \eps|\,| \xi' \eps'|}
                  \ \le \ \frac{k}{c_1c_2}.$$
Now we proceed as in the complex case, put $\beta = \xi \eps$,
and choose $c_1, c_2$ in such a way that the resulting bounds on 
$X = |\beta(\alpha'' - \alpha')|$, $Y = |\beta'(\alpha - \alpha'')|$ and
$Z = |\beta''(\alpha' - \alpha)|$ are equal. In fact, setting
$$ c_1^3 \ = \ k\,\frac{{\gamma'_1} {\gamma'_2} 
	|\alpha-\alpha'| |\alpha-\alpha''|}{{\gamma}_1^2
		{\gamma}_2^2|\alpha'-\alpha''|^2 }, \quad 
   c_2^3 \ = \ k\,\frac{{\gamma}_1 {\gamma}_2 
	|\alpha-\alpha'| |\alpha'-\alpha''|}{{\gamma'}_1^2
		{\gamma'}_2^2 |\alpha-\alpha''|^2}$$
yields the bounds
$X, Y, Z \le \sqrt[3\,]{k\cdot \gamma^{\phantom{p}}_1 \gamma^{\phantom{p}}_2
	{\gamma'_1} {\gamma'_2}} \sqrt[6\,]{\Delta}.$
In particular, we find
$$ |c| \ \le \ \frac{1}{\sqrt{\Delta}}(X + Y + Z) \ \le \ 
	3 \sqrt[3\,]{k\cdot \gamma^{\phantom{p}}_1 
	\gamma^{\phantom{p}}_2	{\gamma'_1} {\gamma'_2}/\Delta}.$$
Making use of Lemma \ref{Lin} we can improve this by a factor
of almost $3/2$. The bounds for $|b|$ and $|a|$ are derived
similarly; this concludes the proof of Prop. \ref{P1} in the
cubic case. For general number fields, the proof makes
use of the dual basis (cf. the proof of Prop. \ref{P0}).

Let us get back to our example of real cubic field with 
discriminant $d = 985$.
Let $\alpha$ denote a root of $f$. Then $\{1, \alpha, \alpha^2\}$
is an integral basis of $\OO_K$, and two fundamental units
are given by $\eta_1 = \alpha$ and $\eta_2 = 2-\alpha$.
Put $\xi_0 = \frac15(2-\alpha+2\alpha^2)$; then 
$\Orb(\xi_0) = \{\xi_0\}$. Using $k = 1.05$ we get the
bounds $\mu_1 = 6.2$, $\mu_2 = 3.2$, $\mu_3 = 1.9$. 
We find $M(K,\xi_0) = N(\xi_0-2) = 1$.

\medskip
\noindent
{\tt Eu3\_5} is the program which does these computations.
In fact, for each $\xi$ in our list {\tt disc.n}, it calculates
$$\min \ \{N(\xi_j+a+b\alpha+c\theta)\}  $$
for all $(a,b,c)\in\Z^3$ such that the coefficients of the
sum $\xi_j+a+b\alpha+c\theta$ are smaller than the bounds 
$\mu_0, \mu_1, \mu_2$ computed in Prop. \ref{P1}.
Since the constants $\mu_j$ depend on the number
$k$, we have to rerun the program replacing $k$ by a real number
$k_1$ larger than the conjectured minimum (i.e. the one obtained
by running {\tt Eu3\_5} on it). The biggest value obtained 
from the various exceptional points gives the
Euclidean minimum $M(K)$ unless they are all smaller than $k$.

The situation is, however, not always as simple as 
in Thm. \ref{T1}. In fact, looking once more at 
the cubic number field $K$ with discriminant $985$ and
using our programs with $k = 0.39$, we can cover $\overline{F}_+$
except for

\begin{center}
\begin{tabular}{|l|ccccc|} \hline
$T_1$ & $[0.345,0.35]$    & $\times$ & $[-0.4915,-0.49]$   
	& $\times$ & $[-0.0185,-0.018]$ \\
$T_2$ & $[0.0175,0.0185]$ & $\times$ & $[-0.2375,0.-2355]$
	& $\times$ & $[0.4725,0.475]$ \\
 $T$  & $[0.3995,0.4005]$ & $\times$ & $[-0.201,-0.199]$ 
	& $\times$ &  $[0.3995,0.4005]$ \\
$T_3$ & $[0.4725,0.473]$  & $\times$ & $[-0.146,-0.1445]$
	& $\times$ & $[0.2905,0.2915]$ \\
$T_4$ & $[0.2905,0.2915]$ & $\times$ & $[0.217,0.219]$ 
	& $\times$ & $[-0.437,-0.436]$ \\
$T_5$ & $[0.436,0.437]$   & $\times$ & $[0.3265,0.3285]$ 
	& $\times$ & $[0.345,0.346]$ \\ \hline \end{tabular}
\end{center}

Applying Thm. \ref{T1} to $T$ shows that 
$\xi = \frac15(2-\alpha+2\alpha^2)$ is the only possible
$k$-exceptional point in $T$. 
Letting $\eps = \alpha$ act on the $T_i$ we find
that the `orbit' of $T_1$ is $\{T_1, -T_2, -T_3, -T_4, T_5\}$;
at this point we need

\begin{cor}\label{C1}
Let $K$ be a number field, $\eps \in E_K$ a non-torsion unit,
and suppose that $T_1$, \ldots, $T_t$ are compact subsets of $\Ft$
with the following properties:
\begin{enumerate}
\item there exist $\beta_1, \ldots, \beta_t \in \OO_K$ such that, for all 
      $\xi \in T_j$, $\eps\xi - \beta_j$ lies in a $k$-covered
      region of $\Ft$ or in $T_{j+1}$ (we put 
      $T_{t+1} = T_1$);
\item for all $\xi \in T_j$ there is a $\gamma \in \OO_K$
      such that $\eps^{-1}\xi - \gamma$ lies in a $k$-covered
      region of $\Ft$ or in $T_{\pi(j)}$, where
      $\pi(j)$ is an index depending only on $j$ and not on $\xi$; 
\item $|\eps|_j \ne 1$ for $1 \le j \le r+s$.
\end{enumerate}
Then the only possible $k$-exceptional
point of $T_1$ is $\zeta = \frac{\beta}{\eps^t-1}$, where
$\beta = \eps^{t-1}\beta_1 + \eps^{t-2}\beta_2 + 
	\ldots + \eps\beta_{t-1} + \beta_t$. Moreover,
the only possible $k$-exceptional points of the sets $T_j$ are
contained in $\Orb_\eps(\zeta)$.
\end{cor}

\begin{proof}
Suppose that $\xi_1 \in T_1$ is $k$-exceptional; then
$\xi_2 = \eps\xi_1 - \beta_1 \in T_2$, $\ldots$, 
$\xi_t = \eps\xi_{t-1} - \beta_{t-1} \in T_t,$ and 
$\xi_{t+1} = \eps\xi_t - \beta_t \in T_1$ are $k$-exceptional.
Observe that
$\xi_{t+1}= \eps\xi_t - \beta_t = \eps^2\xi_{t-1} - 
	\eps\beta_{t-1}-\beta_t = \ldots =  
     \eps^t\xi_1 - \beta$.
From the assumptions made we can deduce that, for every
$k$-exceptional point $\xi_1 \in T_1$, $\eps^t\xi_1 - \beta$
lies again in $T_1$. This shows that condition 1. of Thm. \ref{T1}
is satisfied with $\eps$ replaced by $\eps^t$.

In order to prove that condition 2. is also satisfied we
use induction to find that there exists a $\gamma \in \OO_K$
such that $\xi = \eps^{-t}\xi_1-\gamma$ is an element of some 
set $T_i$. From what we have proved already we know
that there exists a uniquely determined $\gamma' \in \OO_K$
such that $\eps^t\xi - \gamma' \in T_i$. But
$\eps^t\xi - \gamma' = \xi_1 -\eps^t\gamma - \gamma' \in T_i$
implies that $\eps^t\gamma + \gamma' = 0$ and $i = 1$.
Thus condition 2. of Thm. \ref{T1} is also satisfied,
and we can conclude that $\zeta = \beta/(\eps^t-1)$
is the only possible $k$-exceptional point in $T_1$. 
This in turn implies that $\eps\zeta - \gamma_1$ is the
only possible $k$-exceptional point in $T_2$, etc., and
all our claims are proved.
\end{proof}

In our example of the cubic field of discriminant $d = 985$
and $k = 0.39$ we now check that condition 1. of
Cor. \ref{C1} is satisfied, and we find 
$\beta_1 = 0$, $\beta_2 = -3\alpha+\alpha^2$, 
$\beta_3 = -2\alpha$, $\beta_4 = 2\alpha-\alpha^2$, 
and $\beta_5 = 3\alpha$. This gives
$\beta = 7+41\alpha-28\alpha^2$ and
$\xi_1 = \beta/(\alpha^5-1) = \frac1{55}(19-27\alpha-\alpha^2)$.

After having verified condition 2.,
Cor. \ref{C1} shows that the only possible $k$-ex\-cept\-ion\-al
point of $T_1$ is $\xi_1$. Thus the only $k$-exceptional points
of $K$ in $\bigcup T_i$ are 
$\xi_1 = \frac1{55}(19-27\alpha-\alpha^2)$,
$\xi_2 = \frac1{55}(-1 + 13\alpha - 26\alpha^2)$,
$\xi_3 = \frac1{55}(-26 + 8\alpha - 16\alpha^2)$,
$\xi_4 = \frac1{55}(-16 -12\alpha + 24\alpha^2)$,
$\xi_5 = \frac1{55}(24 + 18\alpha + 19\alpha^2)$.
Using $k = 0.5$ we get the bounds $\mu_1 = 4.9$, $\mu_2 = 2.5$, 
$\mu_3 = 1.5$, and $M(K,\xi_1) = N(\xi_1+\alpha) = \frac5{11}$.

\begin{rem}
Suppose that, in our Example, we apply Prop. \ref{P1}
to $\xi = \frac15(2-\alpha+2\alpha^2)$ with $k = 0.39$;
the minimal norm of the elements satisfying conditions
1. -- 3. turns out to be $1$. Nevertheless we can only conclude
that $M(K,\xi) > 0.39$. In order to prove that $M(K,\xi) = 1$
we have to apply Prop. \ref{P1} again, this time with a $k > 1$.
Again, the minimal norm is $1$, and now we can conclude that
in fact $M(K,\xi) = 1$.
\end{rem}

\begin{rem}\label{R12} 
Collecting the uncovered subcubes $S_j$ into the
sets $T_i$ of Cor. \ref{C1} is done as follows: assume that
$S$ is an uncovered cube, and that 
$\eps S - \beta \in \overline{F}$. Then the set $T_1$ containing
$S$ is taken to be the set of all uncovered $S_i$
`near' to $S$ such that $\eps S_i - \beta \in \overline{F}$.
By proceeding similarly with the uncovered cubes in
$\overline{F} \setminus T_1$ we eventually arrive at
subsets $T_i$ containing all uncovered subcubes. 
\end{rem}

The programs are available from the authors. We remark that
they can also be used to study weighted
norms; details will be presented in \cite{CL}.

\section{Some Heuristic Observations}

Consider some $\xi = \frac{\alpha}{\beta} + \OO_K \in K/\OO_K$; 
the bigger $\# \Orb(\xi)$, the more likely it is that one of the points 
in the orbit can be approximated sufficiently well by some 
$\eta \in \OO_K$. In fact, if $(\alpha,\beta) = (1)$ and if 
$\# \Orb(\xi)$ is maximal (i.e. 
$\# \Orb(\xi) = (\OO_K:\beta\OO_K)^\times$), then clearly
$M(K,\xi) = 1/N\beta$.  Euclidean minima thus tend to be 
attained at points $\xi$ with {\em small} orbits. If $K$ has 
unit rank $1$, then there are many points with small orbit: 
just take any $\frac{\alpha}{\eps-1}$ for $\alpha \in \OO_K$ 
and $\eps \in E_K$ a fundamental unit. If the unit rank is 
$\ge 2$, however, such points are hard to find, because there 
is no guarantee that $\alpha/(\eps_1-1)$ has a small 
orbit with respect to the action of a second unit $\eps_2$.

There is one exception, however: suppose that there is a
principal prime ideal $\fp = (\pi)$ which is completely ramified
in $K/\Q$. Since $\fp$ has degree $1$, for every $\eps \in E_K$ 
there is an integer $a \in \Z$ such that $\eps \equiv a \bmod \fp$.
Taking the norm gives $\pm 1 = N_{K/\Q}\eps \equiv N_{K/\Q}a
= a^n \bmod \fp$, where $n = (K:\Q)$, and this in turn implies
that $\eps^n \equiv a^n \equiv \pm 1 \bmod \fp$. Therefore
the unit group generates at most $2n$ different residue classes
mod $\fp$, hence $\# \Orb(\xi) \le 2n$ for any $\xi$ of type
$\xi = \frac{\alpha}{\pi} + \OO_K$. Therefore such $\xi$ have
comparatively small orbit and a good chance of producing
a large minimum. In fact, almost all known Euclidean minima of 
normal cubic fields are attained at such points.

Another question we would like to discuss is the following:
can we expect that our list of norm-Euclidean complex cubic
number fields is complete?
Let us see what is happening in the real quadratic case. There
we know (cf. \cite{Lem2}) that (in the following,
$d = \disc K$ denotes the discriminant of $K$)
$$ \frac{\sqrt{d}}{16+6\sqrt6} \ \le \ M(K) \ \le \
\frac{\sqrt{d}}{4}.$$
This allows us to define the Davenport constant $D = \sup
M(K)/\sqrt{d}$
for real quadratic fields. The example $K = \Q(\sqrt{13}\,)$,
$M(K) = 1/3$ shows that $D \ge 1/3\sqrt{13}$. If we assume that
this is a good approximation for $D$, then there should be
no norm-Euclidean number fields with discriminants
$> D^{-2} = 9\cdot 13 = 117$; in fact, the maximal
discriminant of a norm-Euclidean number field is $d = 76$. 

If we try to do the same with complex cubic case then
the first problem is that the exponent $1/2$ of the
discriminant in the lower bound in 
$$ \frac{\sqrt{|d|}}{420} \ \le \ M(K) \ \le \
    \frac{|d|^{2/3}}{16\sqrt[3\,]{2}}$$
is not known to be best possible. If it is, then we can
define a Davenport constant $D = \sup M(K)/\sqrt{d}$
for complex cubic fields as well. The example
$d = -244$, where $M(K) = 1/2$, shows that $D \ge 1/2\sqrt{244}$,
and if this bound is good, then there should be no
norm-Euclidean number fields with $|d| > 976$.
The example $d = -503$ suggests that $D$ is somewhat smaller,
but in any case we don't expect to find norm-Euclidean
fields with $|d| > 1500$. Basically the same conclusions 
(with better bounds) hold if the correct exponent of $|d|$
in the lower bound is $2/3$.

In the case of totally real cubic fields there is no known
(nontrivial) lower bound for $M(K)$ at all (of course 
$M(K) \ge \frac18$). If one could show 
$M(K) \ge c \sqrt{d}$ for some $c$, then the above heuristics 
show that one has to compute $M(K)$ at least for fields with 
discriminants up to $25.\,000$ (in fact Godwin and Smith 
\cite{GS} have shown that the normal cubic field with 
discriminant $d = 157^2 = 24.\,649$ is norm-Euclidean); 
our current data are therefore insufficient for deciding 
whether such a lower bound might exist or not.

\section{A conjecture}
We would like to conclude our paper with a conjecture\footnote{Added 
in 2012: This conjecture was meant to hold in the rings $\Z[\alpha]$.
This is the ring of integers in $K$ only if $m \not \equiv \pm 1 \bmod 9$,
so in particular values $\ell \equiv 0 \bmod 6$ must be excluded, and 
the smallest possible values of $\ell$ are $\ell = 4, 10$, and $16$.} 
concerning $M(K)$ for certain pure cubic fields $K$:

\begin{conjecture}
Let $m = \ell^3+1$ be a squarefree integer, and assume that
$\ell$ is even; put 
$\alpha = \sqrt[3\,]{\ell}$, $K =  \Q(\alpha)$,  and 
$\xi = \frac12(1+\alpha+\alpha^2)$. Then 
$$M(\K) = M(K,\xi) = \begin{cases}
	\frac1{64}(18\ell^4-9\ell^3+12\ell^2+12\ell), 
		& \text{ if } \ell \equiv 2 \bmod 4, \\
	\frac1{64}(18\ell^4-9\ell^3+30\ell^2+24\ell-32), 
		& \text{ if } \ell \equiv 0 \bmod 4. \end{cases} $$
\end{conjecture}

It is easy to see that $M(K,\xi)$ has at most the value given above; 
in fact, if $\ell \equiv 2 \bmod 4$, then
$ N(\frac14 \ell^2 + \frac12 + \frac14 \ell \alpha - \frac12 \alpha^2)
    = \frac1{64}(18\ell^4-9\ell^3+12\ell^2+12\ell)$, and
if $ \ell \equiv 0 \bmod 4$, then
$N (\frac14 \ell^2 + \frac12 \ell + \frac12 + 
	(\frac14 \ell - \frac12)\alpha - \frac12 \alpha^2)
	= \frac1{64}(18\ell^4-9\ell^3+30\ell^2+24\ell-32).$
Numerical computations show that the conjecture is true
for  $\ell = 4$ and $\ell = 10$.

\section{Acknowledgements}
The computations were done on various systems at the universities
of Bordeaux, Heidelberg and Trento; the latter also provided 
financial support for mutual visits of the authors, and we are
especially grateful to Ren\'e Schoof and Andrea Caranti for
their help.

Roland Qu\^eme has independently developed programs written in {\tt C++} 
for finding norm-Euclidean cubic number fields and used them for
checking many of our results. 

Paul Voutier has kindly sent us E.~Taylor's Ph.~D. thesis; it turns
out that she used the same embedding of $K$ into $\R^3$ as R.~Qu\^eme,
i.e. she mapped $\alpha \in K$ to 
$(\alpha, {\rm Re}\, \alpha', {\rm Im}\, \alpha')$,
where $\alpha, \alpha', \alpha''$ denote the conjugates of $\alpha$. 

Finally we thank G. Niklasch and the referee for
their careful reading of the manuscript.

\medskip

\newpage
\section{Tables}

Euclidean minima of complex cubic number fields
$$ \begin{tabular}{|r|r|c|c|r|r|c|c|}\hline
  $\disc K$ & &  $M_1(K)$  & $M_2(K)$ &  
		$\disc K$ & &  $M_1(K)$ & $M_2(K)$  \\ \hline
  $-$23  &E&    1/5 & $\ge$ 1/7 &  $-$116  &E&   1/2  &  \\ 
  $-$31  &E&    1/3 & $< 1/4$   &  $-$135  &E&   3/5  &  \\ 
  $-$44  &E&    1/2 & 1/4 	&  $-$139  &E&   1/2  &  \\
  $-$59  &E&    1/2 & 1/4 	&  $-$140  &E&   1/2  &  \\
  $-$76  &E&    1/2 & 1/3 	&  $-$152  &E&   1/2  &  \\
  $-$83  &E&    1/2 & 		&  $-$172  &E&   3/4  &  \\
  $-$87  &E&    1/3 & 		&  $-$175  &E&   3/5  &  \\
 $-$104  &E&    1/2 & 		&  $-$199  &N&    1 & $< 0.47$  \\
 $-$107  &E&    1/2 & 		&  $-$200  &E&   1/2  &  \\
 $-$108  &E&    1/2 & 1/3 	&  $-$204  &E& 61/116  & \\ \hline
  $-$211  &E&   59/106& 		&   $-$283  &H&    3/2  &  \\
  $-$212  &E&    5/8  & 		&   $-$300  &E&   23/30 & \\
  $-$216  &E&     1/2 & 		&   $-$307  &N&    9/8  & 3/4 \\
  $-$231  &E&     7/9 & 		&   $-$324  &E&   23/36 & 7/11 \\
  $-$239  &E&    8/9  & 		&   $-$327  &N&  101/99 & \\
  $-$243  &E&    11/18& 		&   $-$331  &H&    3/2  & \\
  $-$244  &E&    1/2  & 		&   $-$335  &N&     1   &  \\
  $-$247  &E&    5/7  & 		&   $-$339  &N&    9/8 &  1 \\ 
  $-$255  &E&   13/15 & 		&   $-$351  &N&     1   & 9/11 \\
  $-$268  &E&   13/22 & $\ge$ 6/11&   $-$356  &E&    7/8 &  \\   \hline
  $-$364  &N&      9/8   & 	&  $-$451  &E&     41/48 & \\
  $-$367  &N&       1    &   9/13 &  $-$459  &N&      9/8  & \\
  $-$379  &E&    397/648 & $\ge$ 11/18 &  $-$460  &E&     43/50 & 23/30 \\
  $-$411  &E&     17/22  & $\ge$  8/11 &  $-$472  &E&     46/61 & \\
  $-$419  &E&      4/5   & 	&  $-$484  &E&     59/76 & \\
  $-$424  &E&     19/27  & $\ge$ 53/76 &  $-$491  &H&       2   & $\ge$ 1\\
  $-$431  &E&     43/64  & 	&  $-$492  &E&     25/32 & \\ 
  $-$436  &N&     79/78  & 	&  $-$499  &E&     23/27 & \\ 
  $-$439  &N&     17/15  & $\ge$ 1&  $-$503  &E& $\ge$ 307/544&  \\
  $-$440  &E&    737/1090& 	&  $-$515  &E&      4/5  & $\ge$ 11/14 \\ \hline
  $-$516  &E&     36/53   &   	&   $-$628  &E&     625/664 & \\
  $-$519  &E&  44712/45747&   	&   $-$643  &H&      25/16  & \\
  $-$524  &N&      5/4    &   	&   $-$648  &H&       5/4   & \\
  $-$527  &N&     13/7    &   	&   $-$652  &E&      21/23  & \\
  $-$543  &E& $\ge$ 158664/170633 &  &  $-$655  &N&      40/23  & \\
  $-$547  &N&      9/8    &   	&   $-$671  &N&      25/19    & \\
  $-$563  &H&       2     &   	&   $-$675  &N&       9/8  & \\
  $-$567  &N&     25/17   & $\ge$ 19/17 &   $-$676  &H&       7/4  & \\
  $-$588  &H&      5/2    &   	&   $-$679  &N&       9/8   & \\
  $-$620  &N&     13/8    & 5/4  	&   $-$680  &N&      (*)    & \\ \hline
\end{tabular} $$

\vfill \eject
$$ \begin{tabular}{|r|r|c|r|r|c|r|r|c|} \hline
 $\disc K$ & &  $M(K)$  & $\disc K$ & &  $M(K)$ \\ \hline
  $-$687  &E&    937/945 	& $-$751  &H&     25/9 		\\
  $-$695  &N&     25/13 	& $-$755  &N&        1 		\\
  $-$696  &E&    186/199 	& $-$756  &N&    306/293	\\
  $-$707  &N&    271/270  	& $-$759  &N&      11/8		\\
  $-$716  &N&   121/109 	& $-$771  &E&     223/252 	\\
  $-$728  &E&     (\S) 		& $-$780  &N&    499/498 	\\
  $-$731  &H&        2 		& $-$804  &N&   $\ge$ 2771/2568 \\	
  $-$743  &N&        1 		& $-$808  &N& $\ge$ 2031/1964	\\
  $-$744  &E&     992/999 	& $-$812  &N&     44/31 	\\
  $-$748  &N&      62/51 	& $-$815  &E&  24543/25325	\\ \hline
  $-$823  &N&     37/25		& $-$891  &H&       7/2 \\
  $-$835  &N& 110353/106265	& $-$907  &N&   $\ge$ 113/108 	 \\
  $-$839  &N&    25/17 		& $-$908  &N&     227/91 		 \\
  $-$843  &N&    134/131	& $-$931  &H&       7/2 		 \\
  $-$856  &N& $\ge$ 454951/428544 	& $-$932  &N&    68425/56788 	 \\
  $-$863  &N&     29/11 	& $-$940  &N&    407/358 		 \\
  $-$867  &N&   1115/1028	& $-$948  &N&   $\ge$ 2120/1959 	 \\
  $-$876  &E&     353/372	& $-$959  &N&     19/7 		 \\
  $-$883  &N&     49/47		& $-$964  &N&    $\ge$ 132/127 	 \\
  $-$888  &N&   2715/2602 	& $-$971  &N&      829/778 	 \\ \hline
  $-$972  &N&    5/4		& $-$1036     &N& 133/101  	 \\
  $-$972  &N&  179/162    	& $-$1048     &N& 617/488 	 \\
  $-$980  &H&   7/4     	& $-$1055     && 			 \\
  $-$983  &N&   31/11 		& $-$1059     &N& 2381/1854 	 \\
  $-$984  &N& $\ge$ 22367/21296	& $-$1067     &N& $\ge$ 160/121	 \\
  $-$996  &N& $\ge$ 6713/5646	& $-$1068     &N& $\ge$ 1499/1350	 \\
  $-$999  & & 			& $-$1075     &N&   777/680  	 \\
  $-$1004 &N&    3167/2298 	& $-$1080     &N& $\ge$ 10253/1000 \\
  $-$1007 &N&        41/23	& $-$1083     &H&   3/2 		 \\
  $-$1011 &N&       271/207  	& $-$1087     &N& 15/8 		 \\ \hline
 $-$1096   &N&  $\ge$ 207/199  	& $-$1176   &H&    4/3		   \\
 $-$1099   &H&  47/26	 	& $-$1187   &N&   11/8 		   \\
 $-$1107   &H&  2 	 	& $-$1188   &N&   $\ge$ 22319/14072    \\
 $-$1108   &N& $\ge$ 4995/4384 	& $-$1191   &N&   11/9 		   \\
 $-$1135   &N& 5115/4033	& $-$1192   &H&  265/168 		   \\
 $-$1144   &N& 4867/3222	& $-$1196   &N&  197/94 		   \\
 $-$1147   &N&   136/99	 	& $-$1203   &N&  $\ge$ 4775/4608 	   \\
 $-$1164   &N&  $\ge$ 1064/918  & $-$1207   &N&  13/9 		   \\
 $-$1172   &N&  572/443   	& $-$1208   &N&  845/656		   \\
 $-$1175   &N&  37/13  	 	& $-$1219   &N&  $\ge$ 709/622   \\  \hline
\end{tabular} $$
\vfill \eject

   Euclidean minima of totally real cubic number fields 

$$ \begin{tabular}{|r|r|c|r|r|c|r|r|c|} \hline
  $\disc K$ & &  $M(K)$  & $\disc K$ & &  $M(K)$  
			& $\disc K$  & &  $M(K)$ \\ \hline
    49  &E&   1/7   &   469  &E&   1/2   &  788  &E&     1/2  \\
    81  &E&   1/3   &   473  &E&   1/3   &  837  &E&     1/2  \\
   148  &E&   1/2   &   564  &E&   1/2   &  892  &E&     1/2  \\
   169  &E&   5/13  &   568  &E&   1/2   &  940  &E&     1/2  \\
   229  &E&   1/2   &   621  &E&   1/2   &  961  &E&    16/31 \\
   257  &E&   1/3   &   697  &E&  13/31  &  985  &N&      1   \\
   316  &E&   1/2   &   733  &E&   1/2   &  993  &E&    31/63 \\
   321  &E&   1/3   &   756  &E&   1/2   & 1016  &E&     1/2  \\
   361  &E&   8/19  &   761  &E&   1/3   & 1076  &E&     1/2  \\
   404  &E&   1/2   &   785  &E&   3/5   & 1101  &E&     1/2  \\ \hline
  1129   &E&     1/3   &  1425   &E&     13/15 &  1708   &E&      1/2 \\
  1229   &E&    16/29  &  1436   &E&      1/2  &  1765   &E&     13/20\\
  1257   &E&     9/25  &  1489   &E&     29/43 &  1772   &E&      1/2 \\
  1300   &E&     7/10  &  1492   &E&      1/2  &  1825   &N&      7/5 \\
  1304   &E&     1/2   &  1509   &E&      1/2  &  1849   &E&     22/43 \\
  1345   &N&     7/5   &  1524   &E&      1/2  &  1901   &E&      1/2  \\
  1369   &E&    31/37  &  1556   &E&      3/4  &  1929   &N&       1  \\
  1373   &E&     1/2   &  1573   &E&     19/22 &  1937   &N&       1  \\
  1384   &E&    11/16  &  1593   &E&   $<$ 0.36  &  1940   &E&      1/2 \\
  1396   &E&     1/2   &  1620   &E&      1/2  &  1944   &E&      1/2 \\ \hline
  1957   &H&     2     &   2241   &E&      3/5  &   2636   &E&  1/2 \\
  2021   &E&    1/2    &   2292   &E&      1/2  &   2673   &E& 64/81 \\
  2024   &E&    1/2    &   2296   &E&      1/2  &   2677   &E& 139/224 \\
  2057   &E&    9/11   &   2300   &E&     27/40 &   2700   &E& 83/120 \\
  2089   &E&    1/2    &   2349   &E&    11/18  &   2708   &E&  1/2\\
  2101   &E&    1/2    &   2429   &E&     1/2   &   2713   &E& $<$ 0.5 \\
  2177   &E&    $<$  0.39 &   2505   &E&     5/9   &   2777   &H&  5/3\\
  2213   &E&    1/2    &   2557   &E&     1/2   &   2804   &E&  1/2\\
  2228   &E&    1/2    &   2589   &E&     9/16  &   2808   &E&  1/2\\
  2233   &E&   56/121  &   2597   &H&     5/2   &   2836   &N&  7/4\\ \hline
 2857	&N&   8/5   & 3137	&E&  $<$  0.59  	& 3325	&E&  \\
 2917	&E&   8/13  & 3144	&E&   1/2 	& 3356	&E&  \\
 2920 	&E&  13/20  & 3173	&E&  $<$  0.59  & 3368	&E&   \\    
 2941	&E&   1/2   & 3221	&E&   1/2    	& 3496	&E&  \\
 2981	&E&   1/2   & 3229	&E&   1/2    	& 3508	&E&  \\     
 2993	&E& $<$0.49 & 3252	&E&    		& 3540	&E&  \\
 3021	&E&   1/2   & 3261 	&E&    		& 3569	&E& \\
 3028 	&E&   1/2   & 3281	&E&    		& 3576	&E&  \\
 3124 	&E&   1/2   & 3305 	&N&  13/9   	& 3580 	&E&  \\
 3132	&E&   1/2   & 3316	&E&    		& 3592 	&E& 5/8 \\ \hline
     \end{tabular} $$

\vfill \eject

$$ \begin{tabular}{|r|r|c|r|r|c|r|r|c|} \hline
  $\disc K$ & &  $M(K)$  & $\disc K$ & &  $M(K)$  
			& $\disc K$  & &  $M(K)$ \\ \hline
 3596 &E&    		& 3892 &E&    		& 4104 &E& $<$  0.55 \\    
 3604 &E&    		& 3941 &E&    		& 4193 &N&  7/5  \\
 3624 &E&    		& 3957 &E&    		& 4212 &H&  7/2  \\         
 3721 &E&  121/183  	& 3969 &H&  7/3   	& 4281 &E& $<$  0.7 \\
 3732 &E&    		& 3969 &H&   1  	& 4312 &N&  11/4  \\       
 3736 &E&    		& 3973 &E&  1/2 	& 4344 &E& $<$  0.7   \\
 3753 &E&    		& 3981 &H&  3/2  	& 4345 &N&  7/5     \\
 3873 &E&    		& 3988 &N&  19/8 	& 4360 &N& 41/35  \\
 3877 &E&    		& 4001 &E&  7/9  	& 4364 &E&    \\
 3889 &N&  13/7  	& 4065 &E&  3/5  	& 4409 &E&   \\ \hline
 4481  &E&         & 4729   &N& 149/73 	& 4860   &E&     \\ 
 4489  &E&  53/67  & 4749   &E&     	& 4892   &E&     \\  
 4493  &E&         & 4764   &E&  17/24 	& 4933   &E&    \\       
 4596  &E&         & 4765   &E&  	& 5073   &E& \\
 4597  &E&         & 4825   &E&   	& 5081   &E&    \\       
 4628  &E&         & 4841   &E&   	& 5089   &N&  17/11  \\
 4641  &E&         & 4844   &E&   	& 5172   &E&     \\
 4649  &E&         & 4852   &E&   	& 5204   &E&   \\
 4684  &N&  13/8   & 4853   &E&   	& 5261   &E&   \\
 4692  &E&  $<$  0.7  & 4857   &E&   	& 5281   &N&   1 \\ \hline
 5297     &N&   21/11  	&  	5468	  &E&		& 5629 	  &E&	 \\
 5300	  &E&		&	5477	  &E&		& 5637	  &E&	 \\
 5325	  &E& 		&	5497      &E&		& 5684    &N&   9/2 \\
 5329	  &N&   9/8	&	5521	  &N& 23/7	& 5685    &E&	 \\
 5333	  &E&		&	5529  	  &E&		 &5697	  &E&	 \\
 5353	  &E&		&	5556	  &E&		& 5724	  &E&	 \\
 5356 	  &E&		&	5613	  &E&		& 5741	  &E& \\
 5368 	  &E&		&	5620	  &E&		& 5780	  &E&	 \\
 5369	  &N&   21/19	&	5621	  &E&		& 5821	  &E&	 \\
 5373	  &E&		&	5624	  &E&	& 5853	  &E&	 \\ \hline
 5901 	 &E&		& 6153    &E&     	& 6420 	&E&	 \\
 5912    &E&		& 6184    &E&		& 6452 	&N&  5/4 \\
 5925    &E&		& 6185    &N&   17/15	& 6453 	&E&  \\
 5940    &E&		& 6209    &E&  		& 6508 	&E&	 \\ 
 5980 	 &E&		& 6237	&E&		& 6549 	&E&	\\
 6053 	 &E&		& 6241 	&N&   223/79	& 6556 	&E& \\
 6088 	 &E&		& 6268 	&E&		& 6557 	&E&	\\
 6092 	 &E&		& 6289 	&N&	1	& 6584 	&E& \\
 6108 	 &E&	   	& 6396 	&E&		& 6588 	&E&	\\
 6133  	 &E&		& 6401 	&N&   35/27	& 6601 	&E&	\\ \hline
  \end{tabular} $$  
\vfill \eject
$$ \begin{tabular}{|r|r|c|r|r|c|r|r|c|} \hline
  $\disc K$ & &  $M(K)$  & $\disc K$ & &  $M(K)$  
			& $\disc K$  & &  $M(K)$ \\ \hline
 6616  &E&  	&  6901  &E& 	& 7220  &H&  9/4 \\
 6637  &E&  	&  6940  &E& 	& 7224  &E& \\
 6669  &E&  	&  6997  &E& 	& 7244  &E& \\
 6681  &E&  	&  7028  &E& 	& 7249  &E& \\ 
 6685  &E& 	&  7032  &E& 	& 7273  &N& 973/601\\
 6728  &E&	&  7053  &H&   2 & 7388  &E&\\
 6809  &H&  7/3 &  7057  &E& 	& 7404  &E&\\
 6856  &E& 	&  7084  &E& 	& 7425  &E& \\
 6868  &N&  5/4	&  7117  &E& 	& 7441  &E&\\
 6885  &N& 67/40 &  7148  &E& 	& 7444  &E&\\ \hline
 7453	&E&		&  	7601  	&E&	     	& 7745 &N& 7/5 \\
 7464   &E&		&  	7628 	&E&	     	& 7753 &E& \\
 7465   &N&    1       	& 	7636 	&E&	     	& 7796 &E& \\
 7473   &E&  $<$  0.89	& 	7641 	&E&	     	& 7816 &E& \\
 7481   &N&    1       	& 	7665 	&E&   21/25    	& 7825 &E& \\
 7528   &N&   17/14 	& 	7668 	&E&	     	& 7873 &N& 29/13 \\
 7537   &N&  227/91	& 	7673 	&E&	      	& 7881 &E& \\
 7540   &E&		& 	7700 	&E&	     	& 7892 &E& \\
 7572   &E&		& 	7709	&E&	     	& 7925 &E& \\
 7573   &N&  41/32	& 	7721 	&E&	     	& 7948 &E& \\ \hline
 8017  &E&	    & 8281  &H&   9/7   & 8532  &E& 	\\
 8057  &E&	    & 8285  &E&	        & 8545  &E&	\\
 8069  &H&   9/2    & 8289  &E&	        & 8556  &E& 	  \\
 8092  &E&	    & 8308  &N&  67/50  & 8572  &N&  17/16 \\
 8113  &N&   13/7   & 8372  &E&	        & 8597  &E&   4/5 \\
 8173  &E&          & 8373  &E&         & 8628  &E& \\
 8220  &E&	    & 8396  &E&         & 8637  &E& \\
 8276  &E&	    & 8468  &H&   5/3   & 8680  &E& \\
 8277  &E&	    & 8472  &E& 	& 8692  &N&  11/10 \\
 8281  &H&  23/16   & 8505  &E&	 	& 8713  &E& \\ \hline
 8745  &E& 	   & 8920  &E& 	       	&  9217  &N& 17/11  \\
 8761  &E& 	   & 9044  &E& 	       	&  9281  &E&  \\
 8769  &E&	   & 9045  &E& 	       	&  9293  &E&  \\
 8789  &N&  23/12  & 9073  &N&   7/5   	&  9300  &E& \\
 8828  &E&	   & 9076  &E&	       	&  9301  &H&    2 \\
 8829  &N&   3/2   & 9133  &E&	       	&  9325  &N&  13/8 \\
 8837  &E&	   & 9149  &E&	       	&  9364  &E&  \\
 8884  &E&	   & 9153  &E&		&  9409  &N&  337/97 \\
 8905  &N&   8/5   & 9192  &E&     	&  9413  &E&  \\
 8909  &E&	   & 9204  &E&	 	&  9428  &E&  \\ \hline
 \end{tabular} $$

\vfill \eject

$$ \begin{tabular}{|r|r|c|r|r|c|r|r|c|} \hline
  $\disc K$ & &  $M(K)$  & $\disc K$ & &  $M(K)$  
			& $\disc K$  & &  $M(K)$ \\ \hline
 9460 &E&  		&  9812 &E&  	& 10004  &E& \\
 9517 &E&  		&  9813 &E&  	& 10040  &E& \\
 9565 &E& $\ge$ 4/5  	&  9833 &E&  	& 10069  &E& \\
 9612 &E&  		&  9836 &E&  	& 10077  &E&  \\
 9653 &N& 35/12		&  9869 &E&	& 10164  &N& 27/22 \\
 9676 &E&  		&  9897 &E&  	& 10172  &E&  \\
 9745 &N& 67/23		&  9905 &N& 9/5	& 10200  &E&  \\
 9749 &E&  		&  9937 &E&  	& 10216  &N&  7/4 \\
 9800 &H& 9/5  		&  9980 &E&  	& 10233  &E&  \\
 9805 &E&  		&  9996 &H& 4/3 & 10260  &E&  \\   \hline
10261 &N& 11/7  	& 10540 &E&  		& 10721 &E&  \\
10273 &H& 27/7 		& 10552 &E&  		& 10733 &E&  \\
10292 &E&  9/10 	& 10561 &N& 11/7 	& 10740 &E&  \\
10301 &E&  		& 10580 &E&  		& 10812 &E&  \\
10309 &H& 11/2 		& 10609 &E&  		& 10844 &E&  \\
10324 &E&  		& 10636 &E&  		& 10865 &E&  \\
10333 &N&  1		& 10641 &E&  		& 10868 &E&  \\
10353 &E&  		& 10661 &E&  		& 10889 &H& 13/5 \\
10457 &N& 27/25		& 10664 &E&  		& 10904 &E&  \\
10484 &E&  		& 10712 &E&  		& 10929 &&  \\ \hline
\end{tabular} $$

\medskip

These Tables contain the known Euclidean minima for cubic
number fields of small discriminant. The fields are ordered
by $|\disc K|$; fields with equal discriminant are ordered
as in the number field tables at Bordeaux \\
({\tt file://megrez.ceremab.u-bordeaux.fr/pub/numberfields/{}}). \\
The letter N indicates 
that the field has class number $1$ but is not norm-Euclidean, 
and $H$ that it has class number $> 1$. Moreover, $E$ means
that the field is norm-Euclidean; if no Minimum is given, we
succeeded in covering the fundamental domain with $k = 0.99$.
CPU-times ranged from a few minutes for fields with small 
discriminants to several hours; by far the hardest nut to
crack was $\disc K = 10661$, which took several days.

\medskip

The only fields with $\disc K < 11,000$ whose Euclidean nature is
currently not known are those with discriminants $10929$
and $10941$. We also remark that among the four fields which were
shown to be Euclidean in \cite{GS}, those with discriminants
$11881$, $16129$ and $24649$ are beyond the limits of our tables.

\medskip

\noindent
(*) The Euclidean minimum $M(K)$ for the field with $\disc K = -680$ is
      $$ M(K) = \frac{81956632}{81182612}. $$
(\S)\ The Euclidean minimum $M(K)$ for the field with $\disc K = -728$
	is $$M(K) = \frac{7483645229}{8158377554}.$$

\medskip

\end{document}